\documentclass[11pt,reqno]{amsart}

\usepackage{marginnote,amssymb,mathrsfs,latexsym,verbatim,pdfsync,color,mathabx}
\usepackage[utf8]{inputenc}

\usepackage[colorlinks=true]{hyperref}
\usepackage{mathtools}

\newcommand{\N}{\mathbb{N}}

\newcommand{\R}{\mathbb{R}}

\newcommand{\loc}{\mathrm{loc}}

\def\e{\mathrm{e}}

\def\supp{\operatorname{supp}}

\numberwithin{equation}{section}

\DeclareFontFamily{U}{mathx}{\hyphenchar\font45}
\DeclareFontShape{U}{mathx}{m}{n}{
      <5> <6> <7> <8> <9> <10>
      <10.95> <12> <14.4> <17.28> <20.74> <24.88>
      mathx10
      }{}
\DeclareSymbolFont{mathx}{U}{mathx}{m}{n}
\DeclareFontSubstitution{U}{mathx}{m}{n}
\DeclareMathAccent{\widecheck}{0}{mathx}{"71}
\DeclareMathAccent{\wideparen}{0}{mathx}{"75}

\makeatletter
\newcommand{\leqnomode}{\tagsleft@true}
\newcommand{\reqnomode}{\tagsleft@false}
\makeatother

\newcommand{\dd}{{d}}

\newtheorem{theorem}{\sc \textbf{Theorem}}[section]  
\newtheorem{proposition}[theorem]{\sc \textbf{Proposition}}   
        
\newtheorem{lemma}[theorem]{\sc \textbf{Lemma}}                

\newcommand{\As}{\mathscr{A}}
\newcommand{\Rs}{\mathscr{R}}

\newcommand{\glob}{\mathrm{glob}}

\theoremstyle{remark}

\title[Riesz Transforms for the inverse Gauss measure]{On the Riesz Transforms for the inverse Gauss measure}
\author[T.\ Bruno]{Tommaso Bruno}
\address{Dipartimento di Scienze Matematiche ``Giuseppe Luigi Lagrange'',
  Politecnico di Torino, Corso Duca degli Abruzzi 24, 10129 Torino,
  Italy}
  \curraddr{Department of Mathematics: Analysis, Logic and Discrete Mathematics\\ Ghent University\\
  Krijgslaan 281\\ 9000 Ghent\\ Belgium}
\email{tommaso.bruno@ugent.be}

\author[P.\ Sj\"ogren]{Peter Sj\"ogren}
\address{Mathematical Sciences,  University of Gothenburg and  Mathematical Sciences\\ Chalmers University of Technology  \\ SE - 412 96 G\"oteborg, Sweden}
\email{peters@chalmers.se}

\keywords{Inverse Gauss measure, Riesz transforms, weak type $(1,1)$}
\thanks{{\em Math Subject Classification} 42B20, 47B03 \\
T.\ Bruno was partially supported by the Research Foundation -- Flanders (FWO) through the postdoctoral grant 12ZW120N. P.\ Sj\"ogren is grateful to the Universit\`a di Genova and the Politecnico di Torino for supporting visits which made this work possible.} 

\begin{document}
\begin{abstract}
Let $\gamma_{-1}$ be the absolutely continuous measure on $\R^n$ whose density is the reciprocal of a Gaussian function. Let further $\As$ be the natural self-adjoint Laplacian on $L^2(\gamma_{-1})$. In this paper, we prove that the Riesz transforms associated with $\As$ of order one or two are of weak type $(1,1)$, but that those of higher order are not.
\end{abstract}
\maketitle
\section{Introduction}
The Euclidean space endowed with the measure $\gamma_{-1}$ whose density is the reciprocal of a Gaussian, which we call the inverse Gauss measure, is a toy model of a variety of settings where a theory of singular integral operators has not yet been established.  Therefore, the analysis of this model may provide profitable insights into more general frameworks. In this paper, we focus on the boundedness of the Riesz transforms. 

As a weighted manifold, $(\R^n,\gamma_{-1})$ has constant, negative definite Bakry--\'Emery curvature tensor. By a celebrated result of Bakry on weighted Riemannian manifolds~\cite{Bakry}, the negative lower bound of this tensor governs the $L^p$ boundedness, $1<p<\infty$, of the shifted Riesz transforms associated with the natural weighted Laplacian on the manifold. No general endpoint analogue is known. In addition to this, the natural weighted Laplacian $\As$ on $(\R^n, \gamma_{-1})$ can be seen as a restriction of the Laplace--Beltrami operator on a warped-product manifold whose Ricci tensor is unbounded from below. Not much is known about Riesz transforms on manifolds of unbounded geometry; see for instance~\cite{ACDH, CD} and references therein. We also emphasize that the inverse Gauss measure setting is intimately related to the Gaussian setting, for $\As$ is unitarily equivalent with a translate of the Ornstein--Uhlenbeck operator.

\smallskip

The connection with the Gaussian setting was the principal motivation for which $\gamma_{-1}$ and $\As$ were initially introduced and studied by Salogni~\cite{Salogni}. Then, several endpoint results for imaginary powers and Riesz transforms of $\As$, involving also new spaces of Hardy type, were obtained in~\cite{Bruno} by the first author of the present paper. There, it was proved that if $\lambda \geq 1$, then the shifted first-order Riesz transforms $\nabla (\As+\lambda I)^{-1/2}$ are of weak type $(1,1)$, that is, bounded from $L^1(\gamma_{-1})$ to $L^{1,\infty}(\gamma_{-1})$. This result, and those concerning Hardy spaces, strongly resemble endpoint analogues of the theorem of Bakry mentioned above. However, the problem whether such a shift is indeed necessary was left open.

The aim of this paper is to study the weak type $(1,1)$ of the Riesz transforms of any order of $\As$. In particular, we improve the first-order results by eliminating the shift. This was unexpected, because the shift {is} necessary for the boundedness from Hardy spaces adapted to $\gamma_{-1}$ to $L^1(\gamma_{-1})$, as proved in~\cite{Bruno}. Another surprising fact is that the weak type $(1,1)$ behaviour of the Riesz transforms of $\As$ is completely analogous to that of the Riesz transforms of the Ornstein--Uhlenbeck operator, although the behaviour on corresponding Hardy spaces is different. See~\cite{GMST, PerSor} and~\cite{Bruno, BrunoOU}, respectively.

To be more explicit, let $\alpha=(\alpha_1,\dots, \alpha_n) \in \N^n$ be a nonzero multi-index and let $\Rs_\alpha= \partial^\alpha \As^{-|\alpha|/2}$, which is a Riesz transform of order $|\alpha| = \alpha_1 + \dots + \alpha_n$. In this paper, we prove the following.

\begin{theorem} \label{main_theorem}
Let $\alpha \in \N^n$ be nonzero. Then $\Rs_\alpha$ is bounded from $L^1(\gamma_{-1})$ to $L^{1,\infty}(\gamma_{-1})$ if and only if $|\alpha|\leq 2$.
\end{theorem}

By~\cite[Remark~2.5]{Bruno}, $\Rs_\alpha$ is bounded on $L^p(\gamma_{-1})$ for every $p\in (1,\infty)$ when $|\alpha|=1$. We do not know whether the same holds in general when $|\alpha|>1$, though we strongly expect so; we leave the investigation of this problem to future work.

The paper is devoted to the proof of Theorem~\ref{main_theorem}. We determine the integral kernel of the Riesz transform $\Rs_\alpha$ and split it into a local and a global part. The local part behaves like a classical Calder\'on--Zygmund kernel, and its weak type $(1,1)$ follows by an adaptation of the classical Calder\'on--Zygmund theory, developed in~\cite{Salogni}. This boundedness holds for the Riesz transforms of any order.  The heart of the proof concerns the global part of the kernel, which requires \emph{ad hoc} techniques where the order of the transforms matters. 

\smallskip

The paper is organized as follows. In Section~\ref{Sec:As}, the operator $\As$, its Mehler-type kernel and its Riesz transforms are defined. We then prove the weak type $(1,1)$ of the local parts of the Riesz transforms in Section~\ref{Sec:local}. In Section~\ref{Sec:lemmata}, we prove weak type $(1,1)$ estimates for operators with several kinds of kernels. These estimates will allow us, in Section~\ref{Sec:global}, to complete the proof of the ``if'' part of Theorem~\ref{main_theorem}. In the final Section~\ref{Sec:3oltre}, it is shown that the weak type $(1,1)$ of the Riesz transforms of order higher than two cannot hold.

\bigskip

We now explain some notation. Let $\underline{0}$ denote the null vector $(0,\dots, 0)\in \N^n$. For $\alpha \in \N^n\setminus\{\underline{0}\}$, we write $\partial^\alpha$ for the differential operator $\partial_{x_1}^{\alpha_1} \cdots \partial_{x_n}^{\alpha_n}$. The Lebesgue measure on $\R^n$ will be denoted by $\dd x$ or by $\mathrm{Leb}$.  If $E$ is a measurable set, $\mathbf{1}_E$ will stand for its characteristic function. Given a linear operator $T$ mapping test functions into measurable functions on $\R^n$, we say that a measurable function $K$ defined and locally bounded off the diagonal in $\R^n \times \R^n$ is the integral kernel of $T$ if, for any test function $f$,
\[
Tf(x)= \int_{\R^n} K(x,y) f(y)\, \dd y, \qquad x\notin \supp(f).
\] 
In other words, the integral kernels of our operators will be taken with respect to Lebesgue measure.

We denote by $C<\infty$, or $c>0$, a constant that may vary from place to place but is independent of significant quantities. Given two positive quantities $A$ and $B$, we shall write $A\lesssim B$ or $B\gtrsim A$ if $A\leq C B$. If $A\lesssim B$ and $B\lesssim A$, we write $A\approx B$. The symbols $A\vee B$ and $A \wedge B$ will stand for $\max (A,B)$ and $\min (A,B)$ respectively.

\section{The operator $\As$ and its Riesz transforms}\label{Sec:As}
For $x\in \R^n$, let $\gamma(x)= \pi^{-n/2}e^{-|x|^2}$ and $\gamma_{-1}(x)= \pi^{n/2}e^{|x|^2}$. With a slight abuse of notation, we identify $\gamma$ and $\gamma_{-1}$ with the measures $\gamma(x)\, \dd x$ and $\gamma_{-1}(x)\, \dd x$, respectively. Consider the second order differential operator
\[
\As_0f(x)=-\frac{1}{2}\Delta f(x)  - x\cdot \nabla f(x), \qquad f\in C_c^\infty(\R^n), \: x\in \R^n, 
\]
which is essentially self-adjoint on $L^2(\gamma_{-1})$. Denote with $\As$ its closure. It is known, see~\cite{Salogni} and~\cite{Bruno}, that the $L^2(\gamma_{-1})$-spectrum of $\As$ is the discrete set $\{n,n+1,\dots\}$, and that its eigenfunctions are the functions $(\gamma H_\alpha)_{\alpha \in \N^n}$, where  $H_\alpha$ is an $n$-dimensional Hermite polynomial, i.e.,
\begin{equation}\label{Hermite}
H_\alpha = \bigotimes_{j=1}^n H_{\alpha_j}, \qquad H_{k}(s) =(-1)^k e^{s^2}\frac{d^k}{ds^k}e^{-s^2}, \quad k\in \N, \: s\in \R.
\end{equation}
Recall that for $t>0$, the integral kernel of the operator $e^{-t\As}$ is
\[H_t(x,y) = \frac{e^{-nt}}{\pi^{n/2}(1-e^{-2t})^{n/2}} e^{-\frac{|x-e^{-t}y|^2}{1-e^{-2t}}}, \qquad x,y\in \R^n;\]
see~\cite{Salogni} and~\cite{Bruno} for this and further details about $\As$.

For $b>0$, the kernel of $\As^{-b}$ is given by the subordination formula
\[K_{\As^{-b}}(x,y) = \Gamma(b)^{-1} \int_0^\infty t^{b-1} H_t(x,y)\, \dd t.\]
The kernel of the Riesz transform $\Rs_\alpha=\partial^\alpha \As^{-|\alpha|/2}$ with $\alpha\in \N^n\setminus\{ \underline{0}\}$ is
\[\begin{split}
K_{\Rs_\alpha} (x,y)
&= \pi^{-n/2}\, \Gamma(|\alpha|/2)^{-1} \int_0^\infty t^{\frac{|\alpha|}{2}-1} \frac{e^{-n t}}{(1-e^{-2t})^{n/2}} \, \partial_x^\alpha e^{-\frac{|x-e^{-t}y|^2}{1-e^{-2t}}}\, dt,\end{split}
\]
in the principal value sense. By the change of variables $t= -\log r$, and recalling~\eqref{Hermite}, we obtain that the kernel of $\Rs_\alpha$ is the principal value of
\begin{align}
K_{\Rs_\alpha} (x,y)
&=\frac{1}{ \pi^{n/2}\, \Gamma(|\alpha|/2)} \int_0^1 \frac{r^{n-1}(-\log r)^{\frac{|\alpha|}{2}-1}}{(1-r^{2})^{n/2}} \, \partial_x^\alpha e^{-\frac{|x-ry|^2}{1-r^2}}\, dr\nonumber \\
&= \frac{(-1)^{|\alpha|}}{ \pi^{n/2}\, \Gamma(|\alpha|/2)}  \int_0^1 \frac{r^{n-1}(-\log r)^{\frac{|\alpha|}{2}-1}}{(1-r^{2})^{(n+|\alpha|)/2}}\, H_\alpha\left(\frac{x-ry}{\sqrt{1-r^2}}\right) e^{-\frac{|x-ry|^2}{1-r^2}}\, dr \label{kernelRalpha1} \\
&= \frac{(-1)^{|\alpha|}}{ \pi^{n/2}\, \Gamma(|\alpha|/2)} \, e^{-|x|^2+|y|^2} \int_0^1 \frac{r^{n-1}(-\log r)^{\frac{|\alpha|}{2}-1}}{(1-r^{2})^{(n+|\alpha|)/2}}\,H_\alpha\left(\frac{x-ry}{\sqrt{1-r^2}}\right) e^{-\frac{|rx-y|^2}{1-r^2}}\, dr, \label{kernelRalpha}
\end{align}
the last equality since 
\[
-\frac{|x-ry|^2}{1-r^2}+\frac{|rx-y|^2}{1-r^2} = -|x|^2+|y|^2.
\]

We split $\R^n \times \R^n$ into a local and a global region. For $\delta=1,2$, define
\begin{equation}\label{LandG}
N_\delta = \bigg\{(x,y)\in \R^n\times \R^n \colon |x-y|\leq \frac{\delta}{1+|x|+|y|}\bigg\}
\end{equation}
and $G =N_1^c$. The regions $N_1$ and $N_2$ will be called \emph{local} regions, while $G$ will be the \emph{global} region. As in~\cite{Bruno}, we fix a smooth function $\chi\colon \R^{n}\times \R^n \to \R$ such that
\[\mathbf{1}_{N_1}\leq \chi \leq \mathbf{1}_{N_2},\qquad |\nabla_x \chi(x,y)|+|\nabla_y \chi(x,y)|\leq \frac{C}{|x-y|}\quad \mbox{for every } x\neq y.\]
We also define
\[
K_{\Rs_\alpha,\loc} = \chi K_{\Rs_\alpha}, \qquad K_{\Rs_\alpha,\glob} =  K_{\Rs_\alpha}-K_{\Rs_\alpha,\loc} .\]
We shall denote the operators with kernels (the principal value of) $K_{\Rs_\alpha,\loc} $ and $K_{\Rs_\alpha,\glob} $ by $\Rs_{\alpha,\loc}$ and $\Rs_{\alpha,\glob}$, respectively. Theorem~\ref{main_theorem} will be proved by putting together Proposition~\ref{Prop_local} (boundedness of $\Rs_{\alpha,\loc}$ for every $\alpha$), Proposition~\ref{Prop_glob} (boundedness of $\Rs_{\alpha,\glob}$ when $|\alpha|\leq 2$), and Proposition~\ref{Prop_3} (unboundedness of $\Rs_\alpha$ when $|\alpha|\geq 3$).

\section{The local part of the Riesz transforms}\label{Sec:local}
In this section, we prove the boundedness of $\Rs_{\alpha, \loc}$.  As we shall see, no restriction on $\alpha$ is necessary: indeed, in the local region all the Riesz transforms of $\As$ behave essentially like classical Calder\'on--Zygmund operators.
\begin{proposition}\label{Prop_local}
$\Rs_{\alpha,\loc}$ is bounded from $L^1(\gamma_{-1})$ to $L^{1,\infty}(\gamma_{-1})$ for every $\alpha\in \N^n \setminus\{ \underline{0}\}$.
\end{proposition}
The proof of Proposition~\ref{Prop_local} will be reduced to proving the boundedness of $\Rs_\alpha$ on $L^2(\gamma_{-1})$ and some Calder\'on--Zygmund type estimates of  $K_{\Rs_\alpha,\loc}$. Following~\cite{Salogni}, we say that a linear operator $T$, mapping test functions into measurable functions on $\R^n$, is a local Calder\'on--Zygmund operator if
\begin{itemize}
\item[(a)] $T$ is a bounded operator on $L^q(\gamma_{-1})$ for some $q\in (1,\infty)$;
\item[(b)] $T$ has an integral kernel $K$, defined and locally integrable off the diagonal in $\R^n \times \R^n$, satisfying
\begin{equation}\label{estimates_LCZ}
|K(x,y)|\lesssim \frac{1}{|x-y|^n} ,\qquad  |\nabla_x K(x,y)|+|\nabla_y K(x,y)|\lesssim \frac{1}{|x-y|^{n+1}}
\end{equation}
for all $(x, y)\in N_2$, $x\neq y$.
\end{itemize}
If $\Rs_\alpha$ is a local Calder\'on--Zygmund operator, then $\Rs_{\alpha, \loc}$ is bounded on $L^p(\gamma_{-1})$ for any $p\in (1,\infty)$, and from $L^1(\gamma_{-1})$ to $L^{1,\infty}(\gamma_{-1})$; see e.g.~\cite[Theorem 3.2.8]{Salogni}.

\smallskip

When $|\alpha|=1$, the boundedness of $\Rs_{\alpha,\loc}$ was proved in~\cite[Proposition 6.6]{Bruno}. We now extend that argument to any nonzero $\alpha$. The key tool will be the following lemma, see~\cite[Lemma 6.3]{Bruno}.
\begin{lemma}\label{lemmalocal}
Let $\mu,\nu\geq 0$ be such that $\mu >\nu +1$. Then, for every $(x,y)\in N_2$, $x\neq y$,
\[
\int_0^1 \frac{|x-ry|^\nu }{(1-r^2)^{\frac{n+\mu}{2}}}e^{-\frac{|x-ry|^2}{1-r^2}}\, \dd r \lesssim \frac{1}{|x-y|^{n+\mu-\nu-2}}.
\]
\end{lemma}
We are now in a position to prove Proposition~\ref{Prop_local}.

\begin{proof}[Proof of Proposition~\ref{Prop_local}]
It is enough to prove that $\Rs_\alpha$ is a local Calder\'on--Zygmund operator.

\smallskip

\emph{Step 1}. We prove that $K_{\Rs_\alpha}$ satisfies the estimates~\eqref{estimates_LCZ}. First, observe that by~\eqref{kernelRalpha1}
\begin{align*}
K_{\Rs_\alpha}(x,y) \lesssim & \int_0^{1/2} {r^{n-1}(-\log r)^{\frac{|\alpha|}{2}-1}} \left| H_\alpha\left(\frac{x-ry}{\sqrt{1-r^2}}\right) \right| e^{-\frac{|x-ry|^2}{1-r^2}}\, \dd r \\  & \qquad + \int_{1/2}^1 \frac{1}{(1-r)^{\frac{n}{2}+1}} \left| H_\alpha\left(\frac{x-ry}{\sqrt{1-r^2}}\right)\right| e^{-\frac{|x-ry|^2}{1-r^2}}\, \dd r.
\end{align*}
Since
\[
\left|H_\alpha\left(\frac{x-ry}{\sqrt{1-r^2}}\right)\right|\lesssim  \sum_{a=0}^{|\alpha|}  \left(\frac{|x-ry|}{\sqrt{1-r^2}}\right)^a, 
\]
and $s^a e^{-s^2} \lesssim 1$ for every $a\geq 0$ with a constant depending only on $a$,
\begin{align*}
\int_0^{1/2} {r^{n-1}(-\log r)^{\frac{|\alpha|}{2}-1}} \left| H_\alpha\left(\frac{x-ry}{\sqrt{1-r^2}}\right) \right| e^{-\frac{|x-ry|^2}{1-r^2}}\, \dd r  &\lesssim \int_0^{1/2} {r^{n-1}(-\log r)^{\frac{|\alpha|}{2}-1}} \, \dd r \lesssim 1.
\end{align*}
Moreover, by Lemma~\ref{lemmalocal} we obtain
\begin{align*}
\int_{1/2}^1 \frac{1}{(1-r)^{\frac{n}{2}+1}} \left| H_\alpha\left(\frac{x-ry}{\sqrt{1-r^2}}\right)\right| e^{-\frac{|x-ry|^2}{1-r^2}}\, \dd r
& \lesssim \sum_{a=0}^{|\alpha|}\int_0^1 \frac{|x-ry|^a}{(1-r^{2})^{\frac{n+2 +a}{2}}} e^{-\frac{|x-ry|^2}{1-r^2}}\, \dd r \\
&\lesssim \frac{1}{|x-y|^n}.
\end{align*}
The estimates of the gradients of $K_{\Rs_\alpha}$ can be obtained analogously. Indeed,
\begin{align*}
\left| \partial_{x_j}\left[ H_\alpha\left(\frac{x-ry}{\sqrt{1-r^2}}\right) e^{- \frac{|x-ry|^2}{1-r^2}}\right] \right|\lesssim e^{- \frac{|x-ry|^2}{1-r^2}}  \left( \sum_{a=1}^{|\alpha|} \frac{|x-ry|^{a-1}}{(1-r^2)^{\frac{a}{2}}} +\sum_{a=0}^{|\alpha|}\frac{|x-ry|^{a+1}}{(1-r^2)^{\frac{a}{2}+1}}\right),
\end{align*}
and the same estimates hold for the derivative along $y_j$. The integral over $(0,1/2)$ goes as before, and by Lemma~\ref{lemmalocal}
\begin{align*}
&|\nabla_xK_{\Rs_\alpha}(x,y)| + |\nabla_y K_{\Rs_\alpha}(x,y)| \\
&\qquad  \lesssim  1+ \left( \sum_{a=1}^{|\alpha|}  \int_{1/2}^1 \frac{|x-ry|^{a-1}}{(1-r^{2})^{\frac{n+2 +a}{2}}} e^{-\frac{|x-ry|^2}{1-r^2}}\, dr +\sum_{a=0}^{|\alpha|} \int_{1/2}^1 \frac{|x-ry|^{a+1}}{(1-r^{2})^{\frac{n+4 +a}{2}}} e^{-\frac{|x-ry|^2}{1-r^2}}\, dr \right)\\
&\qquad  \lesssim \frac{1}{|x-y|^{n+1}}.
\end{align*}
Thus, the estimates~\eqref{estimates_LCZ} hold for $K_{\Rs_\alpha}$.

\smallskip

\emph{Step 2}. We prove that $\Rs_{\alpha}$ is bounded on $L^2(\gamma_{-1})$. For $\beta \in \N^n$, we denote by $h_\beta$ the normalized Hermite polynomial
\[
h_\beta = 2^{-\frac{|\beta|}{2}}\,(\beta!)^{-1/2}\, H_\beta.
\] 
The functions $\gamma \, h_\beta$ form an orthonormal basis in $L^2(\gamma_{-1})$, see e.g.~\cite[\S 5.5]{Szego}. Since for $b>0$
\[
\As^{-b} f = \sum_{\beta \in \N^n} \frac{1}{(|\beta|+n)^b}  \,(f,\gamma\, h_\beta)_{L^2(\gamma_{-1})}\,  \gamma \, h_{\beta},
\]
and, by the definition of $h_\beta$ and $H_\beta$,
\[
\partial^\alpha (\gamma \, h_\beta )= 2^{\frac{|\alpha|}{2}} (-1)^{|\alpha|+|\beta|}\, \sqrt{\frac{(\beta + \alpha)!}{\beta!}}  \, \gamma \, h_{\beta + \alpha},
\]
we obtain
\begin{align*}
\Rs_\alpha f 
&= \partial^\alpha \As^{-|\alpha|/2} f \\
& =  \sum_{\beta} \frac{1}{(|\beta|+n)^{|\alpha|/2}} \, (f,\gamma\, h_\beta)_{L^2(\gamma_{-1})} \, \partial^\alpha (\gamma h_{\beta})\\
&= 2^{|\alpha|/2}  \sum_{\beta} (-1)^{|\alpha|+|\beta|}\frac{1}{(|\beta|+n)^{|\alpha|/2}} \, (f,\gamma\, h_\beta)_{L^2(\gamma_{-1})} \, \sqrt{\frac{(\beta + \alpha)!}{\beta!}} \, \gamma \, h_{\beta +\alpha}.
\end{align*}
Since the $(\gamma\, h_\nu)$ are orthonormal in $L^2(\gamma_{-1})$, and
\[
  \frac{1}{(|\beta|+n)^{|\alpha|}} \,\frac{(\beta + \alpha)!}{\beta!} \lesssim 1,
\]
where the implicit constant may depend on $\alpha$ and $n$, we conclude
\begin{align*}
\|\Rs_\alpha  f \|^2_{L^2(\gamma_{-1})} & = 2^{|\alpha|} \sum_{\beta}  \frac{1}{(|\beta|+n)^{|\alpha|}}  \frac{(\beta + \alpha)!}{\beta!}\, (f,\gamma\, h_\beta)_{L^2(\gamma_{-1})}^2\\
&\lesssim 2^{|\alpha|} \sum_{\beta}  (f,\gamma\, h_\beta)_{L^2(\gamma_{-1})}^2\\
& \lesssim \|f\|^2_{L^2(\gamma_{-1})}.
\end{align*}
This concludes Step 2 and proves that $\Rs_{\alpha}$ is a local Calder\'on--Zygmund operator. The proposition follows.
\end{proof} 

\section{Some technical lemmata}\label{Sec:lemmata}
In this section, we prove some lemmata concerning the weak type $(1,1)$ of various integral operators with given kernels. The kernels are those appearing in the analysis of the Riesz transforms of order $1$ and $2$ in the next section.
 
We first introduce some notation, which will be used in the rest of the paper. Given a couple $(x,y)\in \R^n\times \R^n$, $x\neq \underline{0}$, we shall write
\[y=y_x + y_\perp,\]
where $y_x$ is parallel to $x$ and $y_\perp$ is orthogonal to $x$. We denote with $r_0$ the unique real number such that $y_x=r_0 x$, and let $\theta=\theta(x,y)\in [0,\pi]$ be the angle between $x$ and $y$.  Then,
\[ r_0= (|y|/|x|)\cos \theta,\quad |r_0|=|y_x|/|x|, \qquad |y_\perp|=|y|\sin \theta.
\]
Observe moreover that
\begin{equation}\label{rx-y}
|rx-y|^2=|r-r_0|^2|x|^2 +|y_\perp|^2,
\end{equation}
that 
\begin{equation}\label{1-r0}
|1-r_0|= |x-y_x|/|x|,
\end{equation}
and that
\begin{align}\label{x-ry}
|x-ry| 
&\leq |x-y_x| + |y_x-r y| \nonumber \\ 
&\leq |x-y_x|+ (1-r)|y_x| +r|y_\perp|.
\end{align}
Notice also that $|x \pm y|\geq |x|\sin\theta$ and that, if $(x,y)\in G$, then $|x-y|\geq \frac{1}{2}(1+|x|)^{-1}$, see e.g.~\cite[Lemma 6.2]{Bruno}.

It will be useful to recall that if $T$ is a linear operator such that 
\[
T = \sum_{m=0}^\infty T_m
\]
for some linear operators $(T_m)_{m\in \N}$, each bounded from $L^{1}(\gamma_{-1})$ to $L^{1,\infty}(\gamma_{-1})$, then 
\begin{equation}\label{somma1infty}
\|T\|_{L^1(\gamma_{-1})\to L^{1,\infty}(\gamma_{-1})} \leq \sum_{m=0}^\infty \| T_m\|_{L^1(\gamma_{-1})\to L^{1,\infty}(\gamma_{-1})} \left(1+\log\left(1+\| T_m\|_{L^1(\gamma_{-1})\to L^{1,\infty}(\gamma_{-1})}\right)\right),
\end{equation}
see~\cite[Lemma 2.3]{SW}.

We now begin with the lemmata giving the weak type $(1,1)$ of some integral operators. The first is nothing but~\cite[Lemma 3.3.4]{Salogni} (see also~\cite[Theorem 1]{GMST}), which we restate for the reader's convenience.

\begin{lemma}\label{nucleoSalogni}
The operator with kernel
\[
K (x,y)=e^{-|x|^2+|y|^2}\left[(1+|x|)^n \wedge (|x|\sin\theta)^{-n}\right]
\]
is bounded from $L^1(\gamma_{-1})$ to $L^{1,\infty}(\gamma_{-1})$.
\end{lemma}

\begin{lemma}\label{nucleoBS}
Let $\mu, \nu \in \R$ be such that $\mu+\nu\geq n-2$ and $\mu\leq n$. Then the operator $T$ with kernel
\[
K(x,y)= \pi^{n/2} e^{-|x|^2+|y|^2}|x|^{-\mu} (1+|x|)^{-\nu}
\]
is bounded from $L^1(\gamma_{-1})$ to $L^{1,\infty}(\gamma_{-1})$.
\end{lemma}

\begin{proof}
Let $f\in L^1(\gamma_{-1})$, $f\geq 0$, and observe that
\[
Tf(x) = e^{-|x|^2} |x|^{-\mu} (1+|x|)^{-\nu} \|f\|_{L^1(\gamma_{-1})}.
\]
Now, consider for  $s>0$ the set
\begin{align*}
A_s &=\left\{x\colon Tf(x)>s\right\} =\left\{x\colon e^{-|x|^2} |x|^{-\mu} (1+|x|)^{-\nu} \|f\|_{L^1(\gamma_{-1})} >s\right\}.
\end{align*}
Let $r=r_s$ be the largest positive solution of the equation
\[
e^{-r^2} r^{-\mu} (1+r)^{-\nu}\|f\|_{L^1(\gamma_{-1})} =s;
\]
clearly such a solution exists unless $A_s$ is empty. Now $A_s \subset B(0,r_s)$ and if $r_s\leq 1$, then
\[
\gamma_{-1}(A_s)\leq \int_{B(0,r_s)}e^{|x|^2}\,\dd x\lesssim  r_s^n \lesssim \frac{1}{s}\|f\|_{L^1(\gamma_{-1})}\, r_s^{n-\mu}\lesssim \frac{1}{s}\|f\|_{L^1(\gamma_{-1})},
\]
since $n-\mu\geq 0$. If instead $r_s\geq 1$,
\[
\gamma_{-1}(A_s)\leq \int_{B(0,r_s)}e^{|x|^2}\,\dd x \lesssim e^{r_s^2} \, r_s^{n-2}\lesssim\frac{1}{s}\|f\|_{L^1(\gamma_{-1})}\, r_s^{n-\mu-\nu-2}\lesssim \frac{1}{s}\|f\|_{L^1(\gamma_{-1})},
\]
since $n-\mu-\nu-2\leq 0$. This completes the proof.
\end{proof}

\begin{lemma}\label{nucleo-Hard2}
Let $\delta>0$. Then, the operator $T$ with kernel
\begin{equation*}
K(x,y)=e^{-|x|^2+|y|^2} e^{-\delta |y_\perp|^2} \,|x| \left(\frac{|y|}{|x|}\right)^{n-1}\mathbf{1}_{G\cap \{|y|\leq 2|x|\}}(x,y)
\end{equation*}
is bounded from $L^1(\gamma_{-1})$ to $L^{1,\infty}(\gamma_{-1})$.
\end{lemma}
\begin{proof}
Let $0\leq f \in L^1(\gamma_{-1})$. For $\underline{0}\neq x\in \R^n$, let $x'= x/|x|$ be its projection on the unit sphere $\mathbb{S}^{n-1}$, which we endow with the standard normalized surface measure $dy'$. Let
\begin{align*}
K_0(x,y)&= K(x,y)\, \mathbf{1}_{\{|y_\perp|\leq 1\}}(x,y)\\
K_m(x,y)&= K(x,y)\, \mathbf{1}_{\{2^{m-1}< |y_\perp|\leq 2^m\}}(x,y), \qquad m=1,2,\dots,
\end{align*}
so that
\begin{equation}\label{eqKm}
K=\sum_{m=0}^\infty K_m,
\end{equation}
and let $T_m$ be the operator whose kernel is $K_m$. Observe that if $|y_\perp| \leq 2^m$, then
\[
 |y'-x'| = 2 \sin (\theta/2) \leq 2\sin \theta \leq \frac{2^{m+1}}{|y|} 
 \]
when $\theta \in [0,\pi/2]$, while if $\theta \in [\pi/2,\pi]$ then
 \[
 |y'+x'| = 2 \sin ((\pi-\theta)/2) \leq 2\sin \theta \leq \frac{2^{m+1}}{|y|} .
 \]
Thus, with a constant $c=c(\delta)>0$
\begin{align*}
T_m f(x) &\lesssim |x|^{2-n} e^{-|x|^2} e^{-c 2^{2m}}\int_{\left\{|y|\leq 2|x|, \; |y'-x'| \wedge |y'+x'|\leq 2^{m+1}/|y| \right\}}|y|^{n-1} f(y) \,e^{|y|^2}\, dy\\
& = |x|^{2-n} e^{-|x|^2} e^{-c 2^{2m}} \int_0^{2|x|} \rho^{2n-2} \left(\int_{\left\{|y'-x'|\wedge |y'+x'|\leq 2^{m+1}/\rho\right\}} f(\rho y')\, dy' \right)e^{\rho^2}\, d\rho\\
&= |x|^{2-n} e^{-|x|^2} e^{-c 2^{2m}} I_m(|x|,x'),
\end{align*}
where the last equality defines $ I_m(|x|,x')$. For  $s>0$ let now $B_{s,m} \subset \mathbb{S}^{n-1}$ be the set of all $x'$ such that the equation
\begin{equation}\label{eqIm}
e^{-c 2^{2m}} r^{2-n} e^{-r^2}I_m(r,x') =s
\end{equation}
admits a positive solution $r$; for $x' \in B_{s,m}$, denote by $r_{s,m}(x')$ the largest solution, which exists since $I_m(r,x')  \lesssim r^{n-1}\| f\|_{L^1(\gamma_{-1})}$. Then
\[
A_{s,m}\coloneqq \{ x\colon T_m f(x)>s\} \subseteq \{ x \colon x'\in B_{s,m},\; |x|\leq r_{s,m}(x')\},
\]
and using~\eqref{eqIm} with $r=r_{s,m}(x')$, we get
\begin{align*}
\gamma_{-1}(A_{s,m})& \lesssim \int_{B_{s,m}} \int_0^{r_{s,m}(x')}r^{n-1}e^{r^2}\, dr \, dx'\\
& \lesssim \int_{B_{s,m}} r_{s,m}(x')^{n-2} e^{ r_{s,m}(x')^2}\, dx'\\
& =\frac{1}{s}\, e^{-c2^{2m}} \int_{B_{s,m}} I_m(r_{s,m}(x'),x')\, dx'\\
&\leq \frac{1}{s}\, e^{-c2^{2m}}  \int_0^\infty \rho^{2n-2} \int_{\mathbb{S}^{n-1}} f(\rho y') \left( \int_{\{|y'-x'|\wedge |y'+x'|<2^{m+1}/\rho\}}\, dx'\right) \, dy' e^{\rho^2}\, d\rho\\
&\lesssim \frac{1}{s}\, e^{-c2^{2m}}  \int_0^\infty \rho^{2n-2}  \left(\frac{2^m}{\rho}\right)^{n-1}\int_{\mathbb{S}^{n-1}} f(\rho y')  \, dy' e^{\rho^2}\, d\rho\\
&\approx \frac{1}{s} \, 2^{m(n-1)} e^{-c2^{2m}} \|f\|_{L^1(\gamma_{-1})}.
\end{align*}
In other words, $\|T_m\|_{L^1(\gamma_{-1})\to L^{1,\infty}(\gamma_{-1})}\lesssim 2^{m(n-1)} e^{-c2^{2m}}$. Because of~\eqref{eqKm} and~\eqref{somma1infty}, the proof is complete.
\end{proof}
We conclude this section with another lemma, which will be involved in the study of the Riesz transforms  of order $2$.
\begin{lemma}\label{BLS}
Let $\delta>0$. The operator with kernel
\[
K(x,y)= e^{-|x|^2+|y|^2} \frac{|x|^{\frac{n+1}{2}}}{|x-y_x|^{\frac{n-1}{2}}}\,  e^{-\delta\frac{|y_\perp|^2 |x|}{|x-y_x|}}\, \mathbf{1}_{\left\{|x| |x-y_x|\geq 1, \; \frac{1}{3}|x|\leq |y_x|<|x| \right\}}(x,y)
\]
is bounded from $L^1(\gamma_{-1})$ to $L^{1,\infty}(\gamma_{-1})$.
\end{lemma}
\begin{proof}
We begin with a series of observations that allow us to make some restrictions. 
\begin{itemize}
\item If $n=1$, then $K(x,y)\leq  e^{-|x|^2+|y|^2}|x|$, and the statement follows from Lemma~\ref{nucleoBS}.

\item If $|y_\perp|>|x|/2$, then
\begin{align*}
K(x,y)
&\lesssim \e^{-|x|^2+|y|^2}\frac{|x|^{\frac{n+1}{2}} }{|x-y_x|^{\frac{n-1}{2}}} \left(\frac{|x-y_x|}{|y_\perp|^2|x|}\right)^{\frac{n-1}{2}} \lesssim e^{-|x|^2+|y|^2} |x|^{2-n},
\end{align*}
and the statement follows again from Lemma~\ref{nucleoBS}.

\item If $|y_\perp|\leq |x|/2$ and $\theta> \pi/6$, so that $|x-y_x|\approx |x|$, then
\begin{align*}
K(x,y)
&\leq  e^{-|x|^2+|y|^2}|x| \,e^{-c |y_\perp|^2} \mathbf{1}_{G\cap \left\{\frac{1}{3} |x|\leq |y|\leq 2|x|\right\}}(x,y),
\end{align*}
and we can apply Lemma~\ref{nucleo-Hard2}.

\item If $\frac{|y_\perp|^2 |x|}{|x-y_x|} >\frac{1}{4}|x|^2$, then
\begin{align*}
K(x,y) 
& \lesssim e^{-|x|^2+|y|^2}  \frac{|x|^n}{(|x||x-y_x|)^{\frac{n-1}{2}}} \, e^{-c|x|^2}\mathbf{1}_{\left\{|x| |x-y_x|\geq 1 \right\}}(x,y)  \lesssim  e^{-|x|^2+|y|^2} |x|^{-n},
\end{align*}
and the conclusion follows from Lemma~\ref{nucleoBS}.
\item By means of a rotation, we can assume that $x$ is in the sector defined by $x_1/|x|>\sqrt{3}/2$, that is, the angle between $x$ and the positive first coordinate axis is less than $\pi/6$. Observe that if $\theta\leq \pi/6$, then $y_1/|y|>1/2$.
\end{itemize}

Summing up all this, if $n\geq 2$ and we let
\begin{align*}
\Omega &=\Bigg\{(x,y)\in \R^n \times \R^n \colon \: \frac{x_1}{|x|}>\frac{\sqrt{3}}{2}, \: \frac{y_1}{|y|}>\frac{1}{2}, \: |x||x-y_x|\geq 1, \: \frac{1}{3}|x|\leq |y_x|< |x|,\:  \\ & \hspace{8cm}  |y_\perp|\leq |x|, \:  \theta \in [0,\pi/6], \: \frac{|y_\perp|^2 |x|}{|x-y_x|} \leq\frac{1}{4}|x|^2 \Bigg\},
\end{align*}
then it is enough to prove the boundedness of the operator $\widetilde T$ whose kernel is
\[
\widetilde K(x,y)= e^{-|x|^2} \frac{|x|^{\frac{n+1}{2}}}{|x-y_x|^{\frac{n-1}{2}}} \, e^{-\frac{|y_\perp|^2 |x|}{|x-y_x|}}\, \mathbf{1}_{\Omega}(x,y),
\]
from $L^1(\mathrm{Leb})$ to $L^{1,\infty}(\gamma_{-1})$. We observe that $\Omega\subset G$ and that $(x,y)\in \Omega$ implies $|y|\approx |x|$, so that $x$ and $y$ stay away from the origin.

For $(x,y)\in \Omega$, we have $|x|-|y_x| = |x-y_x|$, and it follows that
\[
|x|-|y| =  \frac{|x|^2-|y|^2}{|x|+|y|}=  \frac{(|x|-|y_x|)(|x|+|y_x|)- |y_\perp|^2}{|x|+|y|} \geq |x-y_x| \frac{|x|+|y_x|-|x|/4}{|x|+|y|} \gtrsim |x-y_x|;
\]
to estimate $|y_\perp|^2$ here, we used the last inequality in the definition of $\Omega$. In particular, $|y|<|x|$. Since also
\[
|x|-|y| \leq |x|-|y_x|= |x-y_x|,
\]
we obtain 
\begin{equation}\label{x-yapprox}
|x|-|y| \approx |x-y_x|.
\end{equation}
 Therefore, there exists $c>0$ such that
\[
\Omega  \subseteq \left\{(x,y)\colon |x|-|y|\geq \frac{c}{|x|}, \; \frac{1}{3}|x|\leq |y_x|, \; |y|<|x|, \; \theta \in [0,\pi/6], \; \frac{|y_\perp|^2 |x|}{|x-y_x|} \leq\frac{1}{4}|x|^2 \right\}.
\]
Let
\begin{align*}
E_0 = \left\{ (x,y)\colon \frac{|y_\perp|^2 |x|}{|x-y_x|}\leq 1\right\}  \quad \mbox{and} \quad  E_m= \left\{(x,y)\colon 2^{m-1}< \frac{|y_\perp|^2 |x|}{|x-y_x|} \leq 2^m\right\}, \quad m=1,2,\dots
\end{align*}
Set also
\[A_0 = \{x\colon |x|\leq 1 \} \quad \mbox{and} \quad  A_k = \{x \colon 2^{k-1} \leq |x|\leq 2^k \}, \quad k=1,2,\dots\]
We define
\[
\widetilde K_{m,k}(x,y)= \widetilde K(x,y)  \mathbf{1}_{E_m}(x,y) \mathbf{1}_{A_k}(x),
\]
so that
\[
\widetilde K = \sum_{m=0}^\infty \sum_{k=0}^\infty \widetilde K_{m,k}.
\]
Observe that if $(x,y)\in \Omega$ and $x\in A_k$, then $y\in A_{k-2} \cup A_{k-1}\cup A_k$. For $(x,y)\in \Omega \cap E_m$ and $x\in A_k$, we make the transformation
\begin{align}\label{trasfxi}
\xi_1 = \frac{1}{2}|x|^2, \qquad & \xi_j =2^{-m/2 + 2k} \, \frac{x_j}{|x|},\quad j=2,\dots, n,
\end{align}
and analogously
\begin{align}\label{trasfeta}
\eta_1 = \frac{1}{2}|y|^2, \qquad & \eta_j =2^{-m/2+2k}\, \frac{y_j}{|y|},\quad j=2,\dots, n.
\end{align}
Going via the coordinates  $(\xi_1, x_2,\dots, x_n)$, one finds that the Jacobian of the transformation~\eqref{trasfxi} is comparable to $x_1\, 2^{-\frac{m}{2}(n-1)} \,2^{2k(n-1)}\, |x|^{1-n} > 0$, and similarly for~\eqref{trasfeta}. This leads to
\[
\dd x \approx 2^{\frac{m}{2}(n-1)} 2^{-nk} \, \dd \xi, \qquad \dd y \approx 2^{\frac{m}{2}(n-1)} 2^{-nk} \, \dd \eta.
\]
Observe that $|x|-|y|\geq c/|x| $ implies
\[
\xi_1 -\eta_1 = \frac{1}{2}(|x| + |y|)(|x|-|y|) \gtrsim c,
\]
and that
\begin{align*}
|x|-|y| = \sqrt{2} \, (\sqrt{\xi_1} -\sqrt{\eta_1}) \approx \frac{\xi_1 - \eta_1}{\sqrt{\xi_1}} \approx \frac{\xi_1 - \eta_1}{|x|}.
\end{align*}
Since $(x,y)\in \Omega \cap E_m$ and $x\in A_k$, we get for $j= 2,\dots, n$, using also~\eqref{x-yapprox},
\begin{align*}
|\xi_j-\eta_j| 
&\lesssim 2^{-\frac{m}{2} +2k} \sin \theta  = 2^{-\frac{m}{2} +2k} \frac{|y_\perp|}{|y|} \lesssim 2^{k} \sqrt{\frac{|x|-|y|}{|x|}}  \approx  \sqrt{|x|(|x|-|y|)} \approx \sqrt{\xi_1 -\eta_1} .
\end{align*}
In other words, if we write $\xi'$ for $(\xi_2, \dots, \xi_n)$, and similarly for $\eta'$, we obtain $|\xi'-\eta'|\lesssim \sqrt{\xi_1-\eta_1}$. 

These transformations lead us to the operator
\[
\mathcal{T}_{m,k}\varphi(\xi)= \int_{\R^n} \mathcal{K}_{m,k}(\xi,\eta) \varphi(\eta)\, \dd \eta,
\]
where 
\[
\mathcal{K}_{m,k}(\xi,\eta) =\mathbf{1}_{[2^{2k-3} ,2^{2k-1} ]}(\xi_1)\, e^{-2^{m-1}} 2^{\frac{m}{2}(n-1)} \,e^{-2\xi_1}(\xi_1 -\eta_1)^{\frac{1-n}{2}}\, \mathbf{1}_{\left\{|\xi'-\eta'|\lesssim \sqrt{\xi_1-\eta_1}, \; \frac{\xi_1}{9}< \eta_1 < \xi_1-c \right\} }(\xi,\eta).
\]
Indeed, the reader may verify that the operator with kernel $\widetilde K_{m,k}$ is bounded from $L^1(\mathrm{Leb})$ to $L^{1,\infty}(\gamma_{-1})$ if $\mathcal{T}_{m,k}$ is bounded from $L^1(\dd \eta)$ to $L^{1,\infty}(e^{2\xi_1}\, \dd \xi)$, with uniform control in $m$ and $k$ of the quotients between the operator quasi-norms. We thus verify that $\mathcal{T}_{m,k}$ has this boundedness property, uniformly in $m$ and $k$.

By~\cite[Proposition 8]{LiSjogren}, $\mathcal{T}_{m,k}$ is bounded from $L^1(\dd \eta)$ to $L^{1,\infty}(e^{2\xi_1}\, \dd \xi) $, with norm $\lesssim e^{-2^{m-1}} 2^{\frac{n-1}{2}m} $ uniformly in $k$. A similar estimate thus holds for the norm from $L^1(\mathrm{Leb})$ to $ L^{1,\infty}(\gamma_{-1})$ of the operator whose kernel is $\widetilde K_{m,k}$.

Since $y\in A_{k-2} \cup A_{k-1}\cup A_k$ if  $x\in A_k$ and $(x,y)\in \Omega$, we can sum over $k$ and conclude that the operator with kernel  $\widetilde K \mathbf{1}_{E_m}$ is bounded from $L^1(\mathrm{Leb})$ to $L^{1,\infty}(\gamma_{-1})$ with norm controlled by $e^{-2^{m-1}} 2^{\frac{n-1}{2}m}$. Finally, the proof is completed with a summation in $m$ and use of~\eqref{somma1infty}.
\end{proof}

\section{The Global Part of the Riesz trasforms}\label{Sec:global}
\begin{proposition}\label{Prop_glob}
If $|\alpha|\in \{1,2\}$, then $\Rs_{\alpha,\glob}$ is bounded from $L^1(\gamma_{-1})$ to $L^{1,\infty}(\gamma_{-1})$.
\end{proposition}
\begin{proof}
Let $(x,y)\in G$. By~\eqref{kernelRalpha},
\begin{align*}
|K_{\Rs_\alpha}(x,y)|&\leq c(|\alpha|) e^{-|x|^2+|y|^2} \int_0^1 \frac{r^{n-1}(-\log r)^{\frac{|\alpha|}{2}-1}}{(1-r^{2})^{\frac{n+|\alpha|}{2}}}\left|H_\alpha\left(\frac{x-ry}{\sqrt{1-r^2}}\right) \right|e^{-\frac{|rx-y|^2}{1-r^2}}\, dr \\
& \lesssim e^{-|x|^2+|y|^2} \sum_{a=0}^{|\alpha|} \left( \int_0^{1/2} r^{n-1}|x-ry|^a e^{-{|rx-y|^2}}\, \dd r + \int_{1/2}^1  \frac{|x-ry|^a}{(1-r)^{\frac{n+2+a}{2}}}e^{-\frac{|rx-y|^2}{1-r^2}}\, \dd r\right)\\& = e^{-|x|^2+|y|^2}\sum_{a=0}^{|\alpha|} 	\left(K^a_1(x,y) + K^a_2(x,y)\right),
\end{align*}
say. We separate the analyses of $K_1^a$ and $K_2^a$.

\smallskip

\textit{Step 1}. We prove that $e^{-|x|^2+|y|^2} K_1^a(x,y)$ is the kernel of an operator of weak type $(1,1)$.  Since
\begin{align*}
K_1^a(x,y)  \lesssim \int_0^{1/2} r^{n-1}(|x|^a+|y|^a) e^{-{|rx-y|^2}}\, \dd r,
\end{align*}
we consider two cases, depending on the values of the quotient $|y|/|x|$.

First, suppose $|y|/|x|\geq 2$. Then $|rx-y| >3|y|/4$ for $0<r<1/2$, so that
\begin{align*}
K^a_1(x,y) 
&\lesssim \int_0^{1/2} r^{n-1}|y|^a e^{-c|y|^2}\, \dd r \lesssim  |y|^a e^{-c|y|^2} \lesssim |y|^{1-n} \lesssim |x|^{1-n},
\end{align*}
and Lemma~\ref{nucleoBS} applies.

We now consider the case $|y|/|x|<2$. Then 
\begin{align*}
K^a_1(x,y)&\lesssim e^{-|y_\perp|^2} \int_0^{1/2} \left[|r_0|^{n-1} + |r-r_0|^{n-1}\right] |x|^a e^{-|r-r_0|^2|x|^2}\, \dd r\\
&\lesssim e^{-|y_\perp|^2} \bigg[\left(\frac{|y|}{|x|}\right)^{n-1} |x|^a\int_0^{1/2} e^{-|r-r_0|^2|x|^2}\, \dd r + |x|^{a}\int_0^{1/2}|r-r_0|^{n-1} e^{-|r-r_0|^2|x|^2}\, \dd r \bigg]\\
& \lesssim e^{-|y_\perp|^2}|x|^{a-1} \left(\frac{|y|}{|x|}\right)^{n-1} +|x|^{a-n}.
\end{align*}
Since $|x|\geq c $ for $|x|>|y|/2$ and $(x,y)\in G$, and also $a \leq 2$, we can apply Lemmata~\ref{nucleo-Hard2} and~\ref{nucleoBS} and complete Step 1.

\smallskip

\textit{Step 2}. We prove that $e^{-|x|^2+|y|^2}K_2^a(x,y)$ is the kernel of an operator of weak type $(1,1)$.

By~\eqref{rx-y},
\begin{align*}
K^a_2(x,y)
&= \int_{1/2}^1  \frac{|x-ry|^a}{(1-r)^{\frac{n+a+2}{2}}}e^{-c\frac{(r-r_0)^2|x|^2 +|y_\perp|^2}{1-r}}\, \dd r\\
& = \int_{1/2}^1  \left( \textbf{1}_{\{ r_0 \leq 1/3\}} + \textbf{1}_{\{ r_0 \geq 2\}} + \textbf{1}_{\{1/3< r_0 <2\}}\right)(x,y) \frac{|x-ry|^a}{(1-r)^{\frac{n+a+2}{2}}}e^{-c\frac{(r-r_0)^2|x|^2 +|y_\perp|^2}{1-r}}\, \dd r\\
& \eqqcolon K_{2,1}^a(x,y) + K_{2,2}^a(x,y) +  K_{2,3}^a(x,y).
\end{align*}
We prove separately the weak type $(1,1)$ of the operators associated to the kernels $e^{-|x|^2+|y|^2}K_{2,i}^a(x,y)$, $i=1,2,3$.

\smallskip

\noindent \textit{2.1}. If $r_0\leq 1/3$ and $1/2\leq r\leq 1$, then 
\begin{equation}\label{211}
|r-r_0|\approx 1+|r_0|.
\end{equation}
Moreover,  
\begin{equation}\label{212}
 |x-ry|\lesssim (1+|r_0|)|x| +|y_\perp|.
\end{equation}
Therefore, with the change of variables $\frac{(1+|r_0|)^2|x|^2 +|y_\perp|^2}{1-r}=s$,
\begin{align*}
K^a_{2,1}(x,y)&\lesssim  \left[(1+|r_0|)|x| +|y_\perp|\right]^a \int_{1/2}^1 \frac{1}{(1-r)^{\frac{n+a+2}{2}}} e^{-c \frac{(1+|r_0|)^2|x|^2 +|y_\perp|^2}{1-r}}\, \dd r\\& \lesssim \left[(1+|r_0|)^2|x|^2+|y_\perp|^2\right]^{-\frac{n}{2}}	\\
& \lesssim |x|^{-n}.
\end{align*}
Thus, the operator with kernel $e^{-|x|^2+|y|^2}K_{2,1}^a(x,y)$ is of weak type $(1,1)$ by Lemma~\ref{nucleoBS}.

\smallskip

\noindent \textit{2.2}. If $r_0\geq 2$, then~\eqref{211} and~\eqref{212} remain valid if $1/2\leq r\leq 1$ and we can argue as in the preceding case.

\smallskip

\noindent \textit{2.3}. We split the integral defining $K_{2,3}^a$ in a way that depends on the value of $r_0$. Let
\begin{align*}
I_1&= [\tfrac{1}{2},1)\cap \left\{r\colon 1-r>\tfrac{3}{2}|1-r_0| \right\},\\ \qquad I_2 &= [\tfrac{1}{2},1)\cap\left\{r\colon 1-r\leq \tfrac{1}{2}(1-r_0) \vee \tfrac{3}{2} (r_0-1) \right\}, \\
I_3& =[\tfrac{1}{2},1)\cap\left\{ r\colon  |r-r_0|<\tfrac{1}{2}(1-r_0)\right\},
\end{align*}
and for $j=1,2,3$ define
\begin{align*}
K_{2,3,j}^a (x,y) =\textbf{1}_{\{1/3< r_0 <2\}}(x,y) \int_{I_j }\frac{|x-ry|^a}{(1-r)^{\frac{n+a+2}{2}}}e^{-c\frac{(r-r_0)^2|x|^2 +|y_\perp|^2}{1-r}}\, \dd r.
\end{align*}
Thus, $K_{2,3}^a = K_{2,3,1}^a + K_{2,3,2}^a+K_{2,3,3}^a$. It will be useful to observe that, since $r_0 \approx 1$, or equivalently  $|y_x| \approx |x|$, and $\theta<\pi/6$, one has
\[
{|y_\perp|^{-n}} \lesssim (|x|\sin\theta)^{-n}.
\]

\smallskip

We first consider $K^a_{2,3,1}$.  Since
\[
1-r - |r_0-1| \leq r_0-r \leq 1-r + |r_0-1|,
\]
the condition $1-r>\frac{3}{2}|1-r_0|$ implies $|r-r_0|\approx 1-r$. Thus,
\[
|x|^2\,\frac{(r-r_0)^2}{1-r}\approx |x|^2(1-r),
\]
and~\eqref{x-ry} implies
\begin{align*}
|x-ry|  &\leq  |1-r_0| |x|+ (1-r)|y_x| +|y_\perp| \lesssim (1-r)|x| +|y_\perp|  .
\end{align*}
We then get
\begin{align*}
K^a_{2,3,1}(x,y) &\lesssim \int_{1-r>3|1-r_0|/2} (1-r)^{-\frac{n+2}{2}} |x|^a(1-r)^{\frac{a}{2}} e^{-c |x|^2(1-r)} e^{-c\frac{|y_\perp|^2}{1-r}}\, \dd r \\ & \qquad \qquad \qquad + \int_{1-r>3|1-r_0|/2} (1-r)^{-\frac{n+2}{2}} \frac{|y_\perp|^a}{(1-r)^{\frac{a}{2}}} e^{-c |x|^2(1-r)} e^{-c\frac{|y_\perp|^2}{1-r}}\, \dd r \\& \lesssim  \int_{1-r>3|1-r_0|/2} (1-r)^{-\frac{n+2}{2}} e^{-c |x|^2(1-r)} e^{-c\frac{|y_\perp|^2}{1-r}} \, \dd r.
\end{align*}
By the change of variables $|y_\perp|^2/(1-r)=s$, one obtains
\[
K^a_{2,3,1}(x,y)  \lesssim    \frac{1}{|y_\perp|^n}.
\]
We also have
\[
K^a_{2,3,1}(x,y)\lesssim   \int_{1-r>3|1-r_0|/2} (1-r)^{-\frac{n+2}{2}} \, \dd r
 \, \lesssim |1-r_0|^{-\frac{n}{2}}
\approx |x|^{\frac{n}{2}}\,
|x-y_x|^{-\frac{n}{2}},
\]
in view of \eqref{1-r0}. Since $(1+|x|)|x-y_x|\gtrsim 1$ if $|x-y_x|>|y_\perp|$, while $(1+|x|)|y_\perp|\gtrsim 1$ if $|x-y_x|\leq |y_\perp|$ (recall $(x,y)\in G$), we obtain
\[
K^a_{2,3,1}(x,y)\lesssim  (1+|x|)^n.
\]
In other words,
\[
K^a_{2,3,1}(x,y)  \lesssim (1+|x|)^n \wedge{(|x|\sin\theta)^{-n}},
\]
and so the operator with kernel $e^{-|x|^2+|y|^2}K^a_{2,3,1}(x,y)$ is of weak type $(1,1)$ by Lemma~\ref{nucleoSalogni}.

\smallskip

We now consider $K^a_{2,3,2}$.  Here $|1-r_0|\approx |r-r_0|$, since
\[
|r-r_0|=|1-r_0 -(1-r)|\leq |1-r_0|+(1-r)\lesssim |1-r_0|.
\]
If $r_0<1$, then $1-r<\frac{1}{2}(1-r_0)$ and hence
\[
|r-r_0|=|1-r_0 -(1-r)|\geq |1-r_0|-|1-r|\geq \frac{1}{2}|1-r_0|.
\]
If $r_0>1$, then
\[
|r-r_0|=r_0-r >r_0-1 = |1-r_0|.
\]
In both cases, by~\eqref{x-ry},
\begin{align*}
|x-ry| \lesssim |x-y_x|+ |y_\perp|
\end{align*}
since $1-r\lesssim|1-r_0|$, $|y_x|\approx |x|$ and by~\eqref{1-r0}. Thus
\begin{align*}
K^a_{2,3,2}(x,y) &\lesssim (|x-y_x| + |y_\perp|)^{a} \int_0^1(1-r)^{-\frac{n+a+2}{2}} e^{-c\frac{|x-y_x|^2 + |y_\perp|^2}{1-r}}\, \dd r.
\end{align*}
After the change of variables $\frac{|x-y_x|^2 + |y_\perp|^2}{1-r}=s$, we obtain
\begin{align*}
K^a_{2,3,2}(x,y)&  \lesssim  (|x-y_x| + |y_\perp|)^{-n}
 \leq \frac{1}{|x-y|^n} \wedge \frac{1}{|y_\perp|^n} \lesssim (1+|x|)^n \wedge \frac{1}{(|x|\sin \theta)^n}.
\end{align*}
Lemma~\ref{nucleoSalogni} now ends the case of $K^a_{2,3,2}$.

\smallskip

It remains to consider $K^a_{2,3,3}$. Observe first that $K^a_{2,3,3} = K^a_{2,3,3} \mathbf{1}_{\{ 1/3<r_0<1\} }$. Then, the conditions $1/3<r_0<1$ and $|r-r_0|<\frac{1}{2}(1-r_0)$ imply $1-r \approx 1-r_0$. Moreover, by~\eqref{x-ry}, the fact that $|y_x|\approx |x|$ and~\eqref{1-r0} we get
\begin{align*}
|x-r y|& \lesssim (1-r_0)|x| + |y_\perp|.
\end{align*}
Therefore
\begin{align*}
K^a_{2,3,3} (x,y)&\lesssim (1-r_0)^{-\frac{n+a+2}{2}}\left[(1-r_0)|x|+|y_\perp|\right]^a \int_{|r-r_0|<\frac{1-r_0}{2}} e^{-c|x|^2 \frac{(r-r_0)^2}{1-r_0} -c|y_\perp|^2\frac{1}{1-r_0}}\, \dd r\\
& \lesssim \left[ (1-r_0)^{-\frac{n-a+2}{2}}|x|^a + (1-r_0)^{-\frac{n+a+2}{2}} |y_\perp|^a \right] e^{-c\frac{|y_\perp|^2}{1-r_0}} \int_{|r-r_0|<\frac{1-r_0}{2}} e^{-c|x|^2 \frac{(r-r_0)^2}{1-r_0}}\, \dd r .
\end{align*}
Observe now that
\begin{equation}\label{stimaint}
\int_{|r-r_0|<\frac{1-r_0}{2}} e^{-c|x|^2 \frac{(r-r_0)^2}{1-r_0} }\, \dd r \lesssim \frac{\sqrt{1-r_0}}{|x|} \wedge (1-r_0).
\end{equation}
Therefore,
\[
K^a_{2,3,3} (x,y) \lesssim A^a(x,y) + B^a(x,y),
\]
where
\[
A^a(x,y)= (1-r_0)^{-\frac{n-a+1}{2}}|x|^{a-1} e^{-c\frac{|y_\perp|^2}{1-r_0}}  \left( 1 \wedge |x| \sqrt{1-r_0}\right)\mathbf{1}_{\left\{ \frac{1}{3}< r_0<1\right\}} (x,y)
\]
and 
\[
B^a(x,y) = (1-r_0)^{-\frac{n+a+2}{2}} |y_\perp|^a e^{-c\frac{|y_\perp|^2}{1-r_0}}\left( \frac{\sqrt{1-r_0}}{|x|} \wedge (1-r_0)\right).
\]
Observe that
\begin{align*}
B^a(x,y)& \lesssim (1-r_0)^{-\frac{n+a+2}{2}}|y_\perp|^a \,\frac{\sqrt{1-r_0}}{|x|}\left(\frac{\sqrt{1-r_0}}{|y_\perp|}\right)^a  e^{-\frac{c}{2}\frac{|y_\perp|^2}{1-r_0}} \\
& \lesssim (1-r_0)^{-\frac{n+1}{2}}\frac{1}{|x|} \,e^{-\frac{c}{2}\frac{|y_\perp|^2}{1-r_0}}.
\end{align*}
Now, if $|y_x-x|\geq |y_\perp|$, then $|y_x-x|\gtrsim (1+|x|)^{-1}$ since $(x,y)\in G$, and~\eqref{1-r0} implies
\begin{align*}
B^a(x,y) 
&\lesssim (1-r_0)^{-\frac{n+1}{2}}\frac{1}{|x|} \lesssim (1+|x|)^n.
\end{align*}
But if $|y_x-x|<|y_\perp|$, then $|y_\perp|\gtrsim (1+|x|)^{-1}$ since $(x,y)\in G$, and again by~\eqref{1-r0}
\begin{align*}
B^a(x,y)& \lesssim (1-r_0)^{-\frac{n+1}{2}}\left(\frac{1-r_0}{|y_\perp|^2}\right)^{\frac{n+1}{2}}\frac{1}{|x|}\lesssim (1+|x|)^n.
\end{align*}
On the other hand, by the definition of $B^a(x,y)$
\begin{align*}
B^a(x,y)& \lesssim(1-r_0)^{-\frac{n+ a+2}{2}}|y_\perp|^a \left(\frac{1-r_0}{|y_\perp|^2}\right)^{\frac{n+a}{2}} (1-r_0) \lesssim (|x|\sin\theta)^{-n}.
\end{align*}
Lemma~\ref{nucleoSalogni} applies, and the operator with kernel $e^{-|x|^2+|y|^2}B^a(x,y)$ is of weak type $(1,1)$. 

We now estimate $A^a$. Arguing as in the two cases for $B^a$, one can show that
\begin{align}\label{1xn}
A^a(x,y) \lesssim (1+|x|)^n.
\end{align}
Now, by means of~\eqref{1-r0} we rewrite $A^a$ as
\[
A^a(x,y) =  \frac{|x|^{\frac{n+a-1}{2}} }{|x-y_x|^{\frac{n-a+1}{2}}}\, e^{-c\frac{|y_\perp|^2|x|}{|x-y_x|}}  \left( 1 \wedge \sqrt{|x||x-y_x|}\right)\mathbf{1}_{\left\{ \frac{1}{3}< r_0<1\right\}} (x,y).
\]
If $a=0$, then
\begin{align}\label{A10}
A^0(x,y)
& \lesssim  \frac{|x|^{\frac{n-1}{2}} }{|x-y_x|^{\frac{n+1}{2}}}  \left( \frac{|x-y_x|}{|x| |y_\perp|^2}\right)^{n/2}  \sqrt{|x||x-y_x|} \lesssim (|x|\sin\theta)^{-n}.
\end{align} 
If $a=1$, then
\begin{align}\label{A11}
A^1(x,y) \lesssim  \frac{|x|^{\frac{n}{2}} }{|x-y_x|^{\frac{n}{2}}}  \left( \frac{|x-y_x|}{|x| |y_\perp|^2}\right)^{n/2} \lesssim (|x|\sin\theta)^{-n}.
\end{align} 
Because of~\eqref{1xn},~\eqref{A10} and~\eqref{A11}, the operators with kernels $e^{-|x|^2+|y|^2}A^0(x,y)$ and $e^{-|x|^2+|y|^2}A^1(x,y)$ are of weak type $(1,1)$ by Lemma~\ref{nucleoSalogni}.

If $a=2$, then 
\[
A^2(x,y) =  \frac{|x|^{\frac{n+1}{2}} }{|x-y_x|^{\frac{n-1}{2}}} e^{-c\frac{|y_\perp|^2|x|}{|x-y_x|}}  \left( 1 \wedge \sqrt{|x||x-y_x|}\right) \mathbf{1}_{\left\{ \frac{1}{3}< r_0<1\right\}} (x,y).
\] 
If $|x| |x-y_x| \leq 1$, then 
\begin{align*}
A^2(x,y) \lesssim \frac{|x|^{\frac{n+1}{2}} }{|x-y_x|^{\frac{n-1}{2}}} \left(\frac{|x-y_x|}{|y_\perp|^2|x|}\right)^{n/2}\sqrt{|x||x-y_x|}  \lesssim (|x|\sin\theta)^{-n}.
\end{align*}
By this and~\eqref{1xn}, the operator whose kernel is $e^{-|x|^2+|y|^2}A^2(x,y) \textbf{1}_{\{|x| |x-y_x| \leq 1\}}$ is of weak type $(1,1)$ because of Lemma~\ref{nucleoSalogni}.  The operator whose kernel is $e^{-|x|^2+|y|^2}A^2(x,y) \textbf{1}_{\{|x| |x-y_x| \geq 1\}}$ is also of weak type $(1,1)$, by Lemma~\ref{BLS}. This concludes the proof of Step 2 and that of the theorem.
\end{proof}

\section{Unboundedness of the Riesz transforms of order at least three}\label{Sec:3oltre}
In this section, we complete the proof of Theorem~\ref{main_theorem}. The proof of the following proposition is inspired by that of~\cite[Section 5]{GMST}, which we adapt to the current setting. 
\begin{proposition}\label{Prop_3}
If $|\alpha|\geq 3$, then $\Rs_\alpha$ is unbounded from $L^1(\gamma_{-1})$ to $L^{1,\infty}(\gamma_{-1})$.
\end{proposition}
\begin{proof} 
Let $\eta>0$ be large, and write $z$ for the point $(\eta, \dots, \eta)\in \R^n$. For every $x\in \R^n$, denote by $x_z$ the component of $x$ which is parallel to $z$, and by $x_\perp=x-x_z$ the component orthogonal to $z$.

Define the tube
\[
J(z)= \{x\in \R^n \colon |x_\perp|<1, \, \tfrac{4}{3}|z| <|x_z|<\tfrac{3}{2}|z|\}.
\]
We claim that if $\eta$ is sufficiently large, then for every  $r\in (0,1)$, $y\in B(z,1)$ and $x\in J(z)$ one has
\[
\frac{x_i -r y_i}{\sqrt{1-r^2}} \gtrsim |z|, \qquad i=1,\dots, n.
\]
Indeed, for these $r,y,x$,
\[
|(x-ry)_i - (x_z-rz)_i|\leq |x-ry - x_z +rz| \leq 2, \qquad \mbox{and} \qquad x_z -rz =c(r)z,
\]
where $c(r)>1/3$ for $r\in (0,1)$. The claim follows, and with $\eta$ large we also conclude that
\[
H_\alpha\left( \frac{x-r y}{\sqrt{1-r^2}}\right) \gtrsim |z|^{|\alpha|}.
\]
Moreover, if $x\in J(z)$, $y\in B(z,1)$ and $1/4<r<3/4$,
\begin{align*}
e^{-\frac{|rx-y|^2}{1-r^2}} 
&\geq e^{-2\frac{|rx-z|^2}{1-r^2} - 2 \frac{|z-y|^2}{1-r^2}}   \gtrsim e^{-c{|rx-z|^2}}  = e^{-c r^2|x_\perp|^2 - c{|rx_z-z|^2}} \gtrsim e^{-c \left|r|x_z|-|z|\right|^2}.
\end{align*}
Thus, if $\eta$ is sufficiently large, $x\in J(z)$ and $y\in B(z,1)$, then (recall~\eqref{kernelRalpha})
\begin{align}\label{estRalpha}
(-1)^{|\alpha|}K_{\Rs_\alpha}(x,y)
& \gtrsim |z|^{|\alpha|} e^{-|x|^2+ |y|^2} \int_{1/4}^{3/4} e^{- c{\left|r|x_z|-|z|\right|^2}}\, \dd r \gtrsim |z|^{|\alpha|-1}e^{-|x|^2+ |y|^2} .
\end{align}
Take now a function $f\geq 0$ supported in the ball $B(z,1)$ and such that $\|f\|_{L^1(\gamma_{-1})}=1$. By~\eqref{estRalpha}, for $x\in J(z)$ and $\eta$ sufficiently large, we have
\begin{align*}
|\Rs_\alpha f(x)| 
&\gtrsim e^{-|x|^2} |z|^{|\alpha|-1} \gtrsim e^{-\left(\frac{3}{2}|z|\right)^2} |z|^{|\alpha|-1}.
\end{align*}
Since
\[
\gamma_{-1}(J(z)) \gtrsim e^{\left(\frac{3}{2}|z|\right)^2} |z|^{-1},
\]
we conclude that
\begin{align*}
\sup_{s>0} \; s \, \gamma_{-1}\{x\colon |\Rs_\alpha f(x)|>s \} 
&\gtrsim e^{-\left(\frac{3}{2}|z|\right)^2} |z|^{|\alpha|-1} \gamma_{-1} (J(z))\gtrsim \eta^{|\alpha|-2},
\end{align*}
from which the proposition follows.
\end{proof}

\end{document}